\documentclass[12pt]{amsart}

\usepackage{amsfonts,amssymb,amsmath,textcomp}

 2

\newtheorem{thm}{Theorem}[section]
\newtheorem{lem}[thm]{Lemma}
\newtheorem{prop}[thm]{Proposition}
\newtheorem{cor}[thm]{Corollary}
\newtheorem{defi}[thm]{Definition}
\newcommand{\thmref}[1]{Theorem~\ref{#1}}
\newcommand{\lemref}[1]{Lemma~\ref{#1}}
\newcommand{\rmkref}[1]{Remark~\ref{#1}}
\newcommand{\propref}[1]{Proposition~\ref{#1}}

\newtheorem{rmk}[thm]{Remark}

\newcommand{\Z}{{\mathbb Z}}

\newcommand{\N}{{\mathbb N}}
\newcommand{\C}{{\mathbb C}}
\begin{document}

\title
{Some Remarks On Rankin-Cohen Brackets of Eigenforms}
\author{Jaban Meher}

\address{Harish-Chandra Research Institute, 
        Chhatnag Road, Jhunsi,
        Allahabad  211 019,
        India.}
\email{jaban@hri.res.in}

\date{\today}
\maketitle
\begin{abstract}
We investigate the cases for which products of two 
quasimodular or nearly holomorphic eigenforms are
eigenforms. We also generalize the results of Ghate
\cite{ghate1} to the case of Rankin-Cohen brackets.
\end{abstract}
\section{Introduction}
The space of modular forms of fixed weight on the full 
modular group has a basis of simultaneous eigenvectors 
for all Hecke operators. A modular form is called an eigenform if 
it is a simultaneous eigenvector for all Hecke operators. 
A natural question to ask is whether  the product of two 
eigenforms (which may be of different weights) is an eigenform. The 
question was taken up by Duke \cite{duke}  and Ghate \cite{ghate}.
They proved that there are only finitely many cases where this
phenomenon happens. Then a more general question i.e., the 
Rankin-Cohen bracket of two eigenforms was studied by Lanphier and 
Takloo-Bighash \cite{lanphier-ramin}. They also proved that except for 
finitely many cases, the Rankin-Cohen brackets of two eigenforms is not
an eigenform. Recently, Beyerl, James, Trentacoste, Xue \cite{bjtx}
have proved that this phenomenon extends to a certain class of
nearly holomorphic modular forms. More explicitly, they have proved 
that there is only one more case apart from the cases listed in
\cite{duke} and \cite{ghate} for which the product of two nearly
holomorphic eigenforms of certain type is a nearly holomorphic eigenform.

   In this paper, we consider a few more cases of such results.
First, we consider the product of two quasimodular eigenforms. Secondly, we
consider the product of nearly holomorphic eigenforms. Finally, we 
generalize the result of Ghate \cite{ghate1} to the case of Rankin-Cohen
brackets.
\section{Quasimodular forms}
Let $\Gamma= SL_2({\Z})$ be the full modular group and 
${\mathcal H}$  denote the upper half plane. Let $M_k$ be the
space of modular forms of weight $k$ on $\Gamma$.
\begin{defi}
A nearly holomorphic modular form $F$ of weight $k$ and depth
$\le p$ on $\Gamma$ is a polynomial in $1/y$ of degree $\le p$
whose coefficients are holomorphic functions on ${\mathcal H}$ 
with moderate growth, such that 
${(cz+d)}^{-k}F\left(\frac{az+b}{cz+d}\right) = F\vert_k \gamma = F$, 
where 
$\gamma= \left(\!\!\begin{array}{ll}a&b\\ c&d\\ \end{array}\!\!\right)
\in SL_2({\Z})$.  
\end{defi}
Let $\widehat{M}_k^{(\le p)}$ denote the space of such forms.
 We denote by $\widehat{M}_k=\cup_p\widehat{M}_k^{(\le p)}$ the space of 
nearly holomorphic modular form of weight $k$  and 
$\widehat{M}_*=\oplus_k\widehat{M}_k$  the graded ring of all 
nearly holomorphic modular forms on $\Gamma$.
\begin{defi}
A quasimodular form of weight $k$ and depth $\le p$ on $\Gamma$ is the 
constant term of  a nearly holomorphic modular form of weight $k$ and
depth $\le p$ on $\Gamma$. 
\end{defi}
Let $\widetilde{M}_k^{(\le p)}$ denote the space of such forms. Let
$\widetilde{M}_k=\cup_p\widetilde{M}_k^{(\le p)}$ be the space of 
quasimodular forms of weight $k$ and 
$\widetilde{M}_*=\oplus_k\widetilde{M}_k$ the graded ring of all
quasimodular forms on $\Gamma$. Then it is known that $\widetilde{M}_*
={\C}[E_2,E_4,E_6]$. 
Here $E_k(z)= 1-\frac{2k}{B_k}\displaystyle\sum_{m\ge1}\sigma_{k-1}(m)q^m$
is the Eisenstein series of weight $k$,
where $B_k$ is the $k$-th Bernoulli number, $\sigma_{k-1}(m)$ is the 
sum of $(k-1)$-th powers of the positive divisors of $m$, and
$q=e^{2\pi iz}$ with $z\in {\mathcal H}$. For more details on 
quasimodular forms see \cite{123}.\\

 For $f\in \widetilde{M}_k$, define the action of $n^{th}$ Hecke operator
$T_n$ on $f$ by
\begin{equation}\label{operator}
(T_nf)(z)=n^{k-1}\displaystyle\sum_{d\vert n}d^{-k}
\displaystyle\sum_{b=0}^{d-1}f\left(\frac{nz+bd}{d^2}\right).
\end{equation}

Then $T_n$ maps $\widetilde{M}_k$ to $\widetilde{M}_k$.
A quasimodular form is said to be an eigenform if it is an 
eigenvector for all of the Hecke operators $T_n$ for $n\in {\N}$.\\
It is known that the differential operator 
$D=\frac{1}{2\pi i}\frac{d}{dz}$ takes $\widetilde{M}_k$ to 
$\widetilde{M}_{k+2}$.
We have the following proposition which follows 
by a similar argument as done in Proposition $2.4$ and
$2.5$ of \cite{bjtx}.
\begin{prop}\label{commute1}
If $f\in \widetilde{M}_k$, then $(D^m(T_nf))(z)=\frac{1}{n^m}(T_n(D^mf))(z)$, 
for $m\ge 0$. Moreover, we have $D^mf$ is an eigenform for $T_n$
iff $f$ is. In this case, if $\lambda_n$ is the eigenvalue of
$T_n$ associated to $f$, then $n^m \lambda_n$ is the eigenvalue 
of $T_n$ associated to $D^mf$.
\end{prop}
 By comparing the constant coefficients of both sides of the 
equality given in Proposition $2.3$ of \cite{bjtx}, we get  
similar identies for the operator $D$. We now state 
two results which follow the same way as was done in
\cite{bjtx}.
\begin{prop}\label{1}
Suppose that $\{f_i\}_i$ is a collection of modular forms of distinct 
weights $k_i$. Then for $a_i\in {\C}^*$, 
$\displaystyle\sum_{i=1}^{t} a_iD^{(n-\frac{k_i}{2})}(f_i)$
is an eigenform if and only if each $D^{(n-\frac{k_i}{2})}(f_i)$ is
an eigenform where the eigenvalues are the same for any $i$.
\end{prop}
\begin{prop}\label{2}
If $k>l$ and $f\in M_k$, $g\in M_l$ are eigenforms, then for
$r\ge 0$, $D^{(\frac{k-l}{2}+r)}(g)$ and $D^{r}(f)$ do not have 
the same eigenvalues.
\end{prop}
\noindent $Notation:$ For $k\in\{12,16,18,20,22,26\}$, let $\Delta_k$ 
denote the unique normalized cusp form of weight $k$ on $\Gamma$.\\
Using the above propositions and following the method as in 
\cite{bjtx}, we have a result analogous to Theorem $3.1$ of 
\cite{bjtx}.
\begin{thm}\label{consequence}
Let $f\in M_k$ and $g\in M_l$ so that for some $r,s\ge 0$,
$D^rf\in \widetilde{M}_{k+2r}$ and $D^sg\in \widetilde{M}_{l+2s}$ are eigenforms. 
Then $(D^rf)(D^sg)$ is an eigenform only in the following cases.
\begin{enumerate}
\item
The modular cases given in \cite{duke} and \cite{ghate}, namely
\begin{center}
$E_4^2=E_8,~~ E_4E_6=E_{10},~~ E_6E_8=E_4E_{10}=E_{14}$,\\
\vspace{.1cm}
$E_4\Delta_{12}=\Delta_{16},~~ E_6\Delta_{12}=\Delta_{18},~~
E_4\Delta_{16}=E_8\Delta_{12}=\Delta_{20},$\\
\vspace{.1cm}
$E_4\Delta_{18}=E_6\Delta_{16}=E_{10}\Delta_{12}=\Delta_{22},$\\
\vspace{.1cm}
$E_4\Delta_{22}=E_6\Delta_{20}=E_8\Delta_{18}=E_{10}\Delta_{16}
=E_{14}\Delta_{12}=\Delta_{26}$.
\end{center}
\item
$(DE_4)E_4=\frac{1}{2}DE_8$.
\end{enumerate}
\end{thm}
\begin{thm}
Let $f\in \widetilde{M}_k^{\le p}$ and $g\in \widetilde{M}_l^{\le q}$
be eigenforms such that $p,q<k/2$. Then $fg$ is  
an eigenform only in the following cases.
\begin{enumerate}
\item
The modular cases given in \cite{duke} and \cite{ghate}, namely
\begin{center}
$E_4^2=E_8,~~ E_4E_6=E_{10},~~ E_6E_8=E_4E_{10}=E_{14}$,\\
\vspace{.1cm}
$E_4\Delta_{12}=\Delta_{16},~~ E_6\Delta_{12}=\Delta_{18},~~
E_4\Delta_{16}=E_8\Delta_{12}=\Delta_{20},$\\
\vspace{.1cm}
$E_4\Delta_{18}=E_6\Delta_{16}=E_{10}\Delta_{12}=\Delta_{22},$\\
\vspace{.1cm}
$E_4\Delta_{22}=E_6\Delta_{20}=E_8\Delta_{18}=E_{10}\Delta_{16}
=E_{14}\Delta_{12}=\Delta_{26}$.
\end{center}
\item
$(DE_4)E_4=\frac{1}{2}DE_8$.
\end{enumerate}
\end{thm}  
\begin{proof}
 We know from Proposition $20$ of \cite{123} (page $59$) that if $p<k/2$, 
then $\widetilde{M}_k^{\le p}=\displaystyle\oplus_{r=0}^{p}D^r(M_{k-2r})$.
Now, if $f\in \widetilde{M}_k^{\le p}$ and $g\in \widetilde{M}_l^{\le q}$
are eigenforms, then by \propref{1} and \ref{2}, we can conclude that
$f=D^r(f_r)$ and $g=D^s(g_s)$, for some $r$, $s$ and $f_r\in M_{k-2r}$, 
$g_s\in M_{k-2s}$. By applying the previous theorem, the result follows.
\end{proof}
\begin{rmk}
{\rm It is known from \cite{123} that if $f$ is a non-zero quasimodular 
form of weight $k$ and depth $p$, then $p\le k/2$.}
\end{rmk}
\begin{rmk}\label{rkm1}
{\rm If $f= \displaystyle\sum_{n\ge 1}a_nq^n \in \widetilde{M}_k$ 
is a non-zero eigenform, then $a_1\ne 0$.
Thus, it follows that the product of two quasimodular eigenforms
(having zero constant term) is not an eigenform.}
\end{rmk}
It is easy to see that $E_2$ is an eigenform.
\begin{rmk}\label{rmk-classify}
{\rm Following the same proof as in the case of $M_k$, one can prove
that a quasimodular form in $\widetilde{M}_k$ with non-zero
constant Fourier coefficient is an eigenform iff $f\in {\C}E_k$.} 
\end{rmk}
We have the following theorem. 
\begin{thm}\label{mainq}
Let $f\in \widetilde{M}_k$ and $g\in \widetilde{M}_l$ be eigenforms
such that the constant coefficients of both $f$ and $g$ are non-zero. Then
$(D^rf)(D^sg)$ is an eigenform only in the following cases.
$$
E_4^2=E_8, ~E_4E_6=E_{10}, ~E_6E_8=E_4E_{10}=E_{14}, 
~(DE_4)E_4=\frac{1}{2}DE_8.
$$
\end{thm}
To prove the above theorem, we first prove the following proposition.
\begin{prop}\label{main1}
Let $f\in M_k$ be an eigenform. Then $E_2f$ is an
eigenform if and only if  $f\in {\C}\Delta_{12}$. 
\end{prop}
\begin{proof}
Since $D\Delta_{12}=E_2\Delta_{12}$, by \propref{commute1}, $E_2\Delta_{12}$ 
is an eigenform.\\
 Conversely, suppose that  $E_2f$ is an eigenform with eigenvalues $\beta_n$,
where $f=\displaystyle\sum_{m\ge 0}a_mq^m\in M_k$ is an 
eigenform with eigenvalues $\lambda_n$. We know that 
$g=Df-\frac{k}{12}E_2f\in M_{k+2}$.
Then $T_n(Df)-\frac{k}{12}T_n(E_2f)=n\lambda_nDf-\frac{k}{12}n\lambda_nE_2f
+\frac{k}{12}(n\lambda_n-\beta_n)E_2f\in M_{k+2}$ . 
Since $E_2f$ is not a modular form and 
$n\lambda_nDf-\frac{k}{12}n\lambda_nE_2f$
is a modular form, we have $n\lambda_n=\beta_n$
for all $n\ge 1$. Thus $g=Df-\frac{k}{12}E_2f\in M_{k+2}$ is an eigenform 
with eigenvalues $n\lambda_n$. \\
If $f=E_k$, then $g=\alpha E_{k+2}$ for some $\alpha \in {\C}$. 
Therefore, by applying $T_n$ to
$\alpha E_{k+2}=DE_k-\frac{k}{12}E_2E_k$, we get for all $n\ge 1$,~  
$n\sigma_{k-1}(n)=\sigma_{k+1}(n)$, which is not true.\\
If $f$ is a cusp form, without loss of
generality assume that $f$ is normalized. Let 
$g=\displaystyle\sum_{m\ge 1}b_mq^m$.
Since $b_1=1-\frac{k}{12}$, we have 
\begin{equation}
b_n=na_n\left(1-\frac{k}{12}\right),
\end{equation}
for all $n\ge 1$.
Now computing the values of $b_n$ from $Df-\frac{k}{12}E_2f$
in terms of $a_n$ and then substituting in the previous equation,
we see that $a_2=-24$, $a_3=252$ and $a_4=-1472$. These are nothing 
but the second, third and fourth Fourier coefficients of $\Delta_{12}$
respectively. But Theorem $1$ of \cite{ghitza} says that if $f_1$ and 
$f_2$ are two cuspidal eigenforms on $\Gamma_0(N)$ of different weights, 
then there exists $n\le 4(\log(N)+1)^2$ such that $a_n(f_1)\ne a_n(f_2)$. 
Applying this theorem to $f_1=f$, $f_2=\Delta_{12}$ and $N=1$, we conclude 
that $k=12$. Thus we have $f=\Delta_{12}$.
\end{proof}
\begin{rmk}\label{rmk2}
 {\rm Since $DE_2=\frac{E_2^2-E_4}{12}$ and $DE_2$, $E_4$ are eigenforms 
with different eigenvalues, $E_2^2$ is not an eigenform.}
\end{rmk}
\begin{proof}[Proof of \thmref{mainq}]
By \thmref{consequence}, \rmkref{rkm1}, \rmkref{rmk-classify}, 
\propref{main1} and \rmkref{rmk2}, we only have to prove 
that in the following cases $(D^rE_2)(D^sE_k)$ is not an eigenform.
\begin{enumerate}
\item
$r=0$ and $s\ge 1$
\item
$r\ge 1$ and $s=0$.
\end{enumerate}
For $(1)$, let us assume on the contrary that $E_2(D^sE_k)$ is an 
eigenform where 
$s\ge 1$. The first few coefficients of the normalized form 
$\frac{-B_k}{2k}E_2(D^sE_k)=\displaystyle\sum_{n\ge 1}a_nq^n$ are 
$a_1=1$, $a_2=2^s\sigma_{k-1}(2)-24$, 
$a_3= 3^s\sigma_{k-1}(3)-24(2^s\sigma_{k-1}(2)+3)$,
$a_4= 4^s\sigma_{k-1}(4)-24(3^s\sigma_{k-1}(3)+3\cdot 2^s\sigma_{k-1}(2)+4)$.\\
Since $\frac{-B_k}{2k}E_2(D^sE_k)$ is an eigenform
we have $a_4=a_2^2-2^{k+2s+1}$ and $a_6=a_2a_3$. Thus we have
$$
4^s\sigma_{k-1}(4)-24(3^s\sigma_{k-1}(3)+3\cdot 2^s\sigma_{k-1}(2)+4)
=2^{2s}\sigma_{k-1}(2)^2-48\cdot2^s\sigma_{k-1}(2)+576-2^{k+2s+1}
$$
and\\
$
6^s\sigma_{k-1}(6)-24(5^s\sigma_{k-1}(5)+
3\cdot4^s\sigma_{k-1}(4)+4\cdot 3^s\sigma_{k-1}(3)+ 
7\cdot2^s\sigma_{k-1}(2)+6)=
(2^s\sigma_{k-1}(2)-24)(3^s\sigma_{k-1}(3)-24(2^s\sigma_{k-1}(2)+3)).
$
\\
From the multiplicativity of $\sigma_{k-1}$ and 
$\sigma_{k-1}(4)=\sigma_{k-1}(2)^{2}-2^{k-1}$, these simplify to

\begin{equation}\label{negative}
3^s(1+3^{k-1})+2^s+28=2^{k+s-4}(2^s-2^3)
\end{equation}
and
\begin{equation}\label{b6}
5^s\sigma_{k-1}(5)+3^{s+1}\sigma_{k-1}(3)
+2^{2s+1}{\sigma_{k-1}(2)}^2+7\cdot 2^{s+2}\sigma_{k-1}(2)
-3\cdot 2^{k+2s-1}+78=0.
\end{equation}
 Now, if $s\le 3$, then the left hand side of \eqref{negative} is 
positive, but the right hand side of the equation is non-positive.
Thus $s\ge 4$.
If $k\equiv 2\pmod{4}$ and $s$ is odd, then 
$7+3^s\left(\frac{1+3^{k-1}}{4}\right)\equiv 2\pmod{4}$, but 
$2^{k+s-6}(2^s-2^3)-2^{s-2}$ is divisible by $4$,
giving a contradiction to \eqref{negative}.
If $k\equiv 2\pmod{4}$ and $s\equiv 0\pmod{4}$, then
$3^s(1+3^{k-1})+2^s(1+2^{k-1})+28 \equiv 0\pmod{5}$,
but $5$ does not divide $2^{k+2s-4}$. This  gives a contradiction.
If $k\equiv 2\pmod{4}$ and $s\equiv 2\pmod{4}$, then 
$3^{s+1}\sigma_{k-1}(3)+2^{2s+1}{\sigma_{k-1}(2)}^2+
7\cdot 2^{s+2}\sigma_{k-1}(2)-3\cdot 2^{k+2s-1}+78 \equiv 4 \pmod{5}$,
but the remaining term of left hand side of \eqref{b6} is divisible 
by $5$, giving a contradiction. If $k\equiv 0\pmod{4}$ and $s$ is even 
or $s\equiv 1\pmod{4}$, then we get a contradiction from 
\eqref{negative} and if $k\equiv 0\pmod{4}$ and $s\equiv 3\pmod{4}$,
we get a contradiction from \eqref{b6}. This proves the theorem for
case $(1)$.\\  
For case $(2)$, let us assume on the contrary that $(D^rE_2)E_k$ 
is an eigenform for $r\ge 1$. 
Let $\frac{-1}{24}(D^rE_2)E_k=\displaystyle\sum_{n\ge 1}b_nq^n$ be the
normalized eigenform. The first few coefficients of the expansion
are 
$b_1=1$, $b_2=3\cdot 2^r-\frac{2k}{B_k}$,
$b_3=4\cdot 3^r-\frac{2k}{B_k}(3\cdot 2^r+ \sigma_{k-1}(2))$,
$b_4=7\cdot 4^r- \frac{2k}{B_k}(4\cdot 3^r+ 3\cdot 2^r\sigma_{k-1}(2)
+ \sigma_{k-1}(3))$.\\
Since $\frac{-1}{24}(D^rE_2)E_k$ is a normalized eigenform, we have
$b_4=b_2^2-2^{k+2r+1}$. Substituting above values of $b_2$ and $b_4$
we get
$$
7\cdot4^r -\frac{2k}{B_k}\left(4\cdot 3^r+3\cdot 2^r\sigma_{k-1}(2)+
\sigma_{k-1}(3)\right)={\left(3\cdot 2^r-\frac{2k}{B_k}\right)}^2-2^{k+2r+1}.
$$
This can be simplified to
$$
{\left(\frac{2k}{B_k}\right)}^2+
\frac{2k}{B_k}\left(4\cdot 3^r+3\cdot 2^r(2^{k-1}-1)+1+3^{k-1}\right)+ 
2^{2r+1}(1-2^k)=0.
$$
\begin{equation}\label{quadratic}
\Rightarrow \frac{2k}{B_k}= \frac{-b \pm \sqrt{b^2+2^{2r+3}(2^k-1)}}{2},
\end{equation}
where 
\begin{equation}\label{b}
b= 4\cdot 3^r+3\cdot 2^r(2^{k-1}-1)+1+3^{k-1}.
\end{equation}
Since $\frac{2k}{B_k}$ is a rational number, $b^2+2^{2r+3}(2^k-1)$ is a 
perfect square, and since $2$ divides $b$, $\frac{2k}{B_k}$ is an 
integer. This implies that $k\in \{2,4,6,8,10,14\}$. Since the case
$k=2$ is shown in case $(1)$, we only consider
$k\in \{4,6,8,10,14\}$. \\

Let $k=4$. In this case, $\frac{2k}{B_k}=-240$. Since 
$\frac{2k}{B_k}$ is negative, from \eqref{quadratic}, we get\\

$-b-\sqrt{b^2+2^{2r+3}(2^4-1)}=-480$\\
$\Rightarrow b^2+15\cdot 2^{2r+3}={(b-480)}^2$\\
$\Rightarrow b= 240-2^{2r-3}$\\ 
Substituting this value of $b$ in \eqref{b}, we get
\begin{equation}\label{case4}
2^{2r-3}+4\cdot 3^r+21\cdot 2^r-212=0.
\end{equation}
Now, we can see that \eqref{case4} is not satisfied for any
positive integer $r$, giving a contradiction. The other cases
are done similarly, whereby one uses \eqref{quadratic} to obtain an
equation in terms of $r$. It is straightforward to show that
this equation cannot be satisfied for any appropriate integer
values of $r$. This concludes the proof of the theorem.
\end{proof}
\begin{cor}
Let $f\in M_k$ be an eigenform. Then $(D^rE_2)f$ is an eigenform 
if and only if $r=0$ and $f\in {\C}\Delta_{12}$.
\end{cor}
\begin{proof}
It is a direct consequence of \rmkref{rkm1}, \thmref{mainq} and
\propref{main1}. 
\end{proof}
\section{Nearly holomorphic modular forms}
\begin{defi}
The Maass-Shimura operator $\delta_k$ on $f\in \widehat{M}_k$ is 
defined by
$$
\delta_k(f)=\left(\frac{1}{2\pi i}\left(\frac{k}{2iIm(z)}+ 
\displaystyle \frac{\partial}{\partial z}\right)f\right)(z).
$$ 
\end{defi}
The operator $\delta_k$ takes $\widehat{M}_k$ to  $\widehat{M}_{k+2}$.
Here we consider the action of $\delta_k$  on $M_k$.
The operator $T_n$, for each $n\ge 1$ as defined by \eqref{operator}, maps 
$\widehat{M}_k$ to $\widehat{M}_k$. The function 
$E_2^*(z)=E_2(z)-\frac{3}{\pi Im(z)}$
is a nearly holomorphic modular form of weight $2$ on $\Gamma$ and 
it is also an eigenform.
\begin{thm}
Let $f$ be a normalized eigenform in $M_k$. Then $E_2^*f$ is an eigenform
if and only if $f=\Delta_{12}$. 
\end{thm}
\begin{proof}
It is known from Proposition $2.5$ of \cite{bjtx} that 
$\delta_{12}(\Delta_{12})=E_2^*\Delta_{12}$ is an eigenform.\\
For any modular form $f\in M_k$, we have
$\delta_k(f)-\frac{k}{12}E_2^*f = Df-\frac{k}{12}E_2f\in M_{k+2}$.
Now assume that $f\in M_k$ is a normalized eigenform  such that $E_2^*f$ 
is an eigenform. Then  proceeding as in \propref{main1}, we conclude 
that $f=\Delta_{12}$.
\end{proof}
\section{Rankin-Cohen Brackets of holomorphic eigenforms}
 
 Let $M_k(\Gamma_1(N))$, $S_k(\Gamma_1(N))$ and $\mathcal{E}_k(\Gamma_1(N))$ be 
respectively the spaces of modular forms, cusps forms and Eisenstein series
of weight $k \ge 1$ on $\Gamma_1(N)$, and let $M_k(N,\chi)$, $S_k(N, \chi)$,
$\mathcal{E}_k(N,\chi)$ be the spaces of modular forms, cusps forms and 
Eisenstein series of level $N$ and character $\chi$ respectively. 
We have an explicit basis $\mathcal{B}$ for $M_k(\Gamma_1(N))$ which consist 
of common eigenforms for all Hecke operators $T_n$ with $(n, N)=1$ as described 
in \cite{ghate1}.
An element of $M_k(\Gamma_1(N))$ is called an almost everywhere eigenform 
or a.e. eigenform for short, if it is constant multiple of an element of  
$\mathcal{B}$. For further details see \cite{ghate1}. 
 
Let $g\in M_{k_1}(N,\chi)$ and $h\in M_{k_2}(N,\psi)$. 
The $m^{th}$ Rankin-Cohen bracket of $f$ and $g$ is defined by
$$
\left[g,h\right]_m(z) = \displaystyle\sum_{r+s=m} (-1)^r 
\displaystyle\binom{m+k_1-1}{s}\displaystyle\binom{m+k_2-1}{r} 
g^{(r)}(z) h^{(s)}(z),
$$
where $g^{(r)}(z)= D^r g(z)$ and $h^{(s)}(z)= D^s h(z)$.\\
It is known that $\left[g,h\right]_m \in M_{k_1 + k_2+ 2m}(N,\chi \psi)$
and $\left[g,h\right]_m$ is a cusp form if $m\ge 1$.\\

For $k > 2$, the Eisenstein series is defined by

$$
E_k^{(N,\psi)}(z)= 
\displaystyle\sum_{\gamma \in \Gamma_{\infty} \diagdown \Gamma_0(N)}
\overline{\psi}(d){\left(cz + d\right)}^{-k}\in \mathcal{E}_k(N,\psi),
$$   
where $z \in {\mathcal H}$
and the sum varies over all $\gamma =
\left(\!\!\begin{array}{ll}a&b\\ c&d\\ \end{array}\!\!\right)\in
\Gamma_0(N)$ modulo\\
$\Gamma_{\infty}= 
\{\left(\!\!\begin{array}{ll}1&n\\ 0&1\\ \end{array}\!\!\right)\vert
n\in {\Z} \}$.
We recall Proposition $6$ of \cite{zagier}:
\begin{thm}\label{zagier}
 Let $k_1$, $k_2$, $m$ be integers satisfying $k_2 \ge k_1 +2>2$ and
let $k=k_1+k_2+2m$. If $f(z)=\displaystyle\sum_{n=1}^{\infty}a_nq^n\in
S_k(N, \chi\psi)$ and $g(z)=\displaystyle\sum_{n=0}^{\infty}b_nq^n\in
M_{k_1}(N,\chi)$, then
$\langle f, [g, E_{k_2}^{(N,\psi)}]_m\rangle = 
\frac{\Gamma(k-1)\Gamma(k_2+m)}{{(4\pi)}^{k-1}m!\Gamma(k_2)}
\displaystyle\sum_{n=1}^{\infty}\frac{a_n \overline{b}_n}{n^{k_1+k_2+m-1}}$,
where $\langle , \rangle$ is the Petersson inner product.
\end{thm}

Now, for an arbitrary positive integer $N$, let $Q\vert N$
such that $(Q,N/Q)=1$. Let $W_Q$ be the Atkin-Lehner operator on
$M_k(N, \chi)$. Let $\chi = \chi_Q \chi_{N/Q} $. Then it is 
known that $W_Q$ maps $M_k(N,\chi_Q \chi_{N/Q})$ to 
$M_k(N,\overline{\chi}_Q \chi_{N/Q})$ and it is an involution. 
It takes cusp forms to cusp forms and a.e. eigenforms to a.e. eigenforms 
(see \cite{li} for details). We have the following lemma 
(see \cite{lanphier}, Proposition 1).
\begin{lem}\label{commute}
If $g\in M_{k_1}(N,\chi)$ and $h\in M_{k_2}(N,\psi)$, then 
$\left[g,h\right]_m\vert W_Q = 
\left[g\vert W_Q, h\vert W_Q\right]_m$.
\end{lem}
Let $\psi_i$ be Dirichlet characters mod $M_i$, $i=1,2$
such that $\psi_1\psi_2(-1)= (-1)^k$, where $k\ge 1$.
Also assume that:
\begin{enumerate}
\item
if $k=2$ and $\psi_1 $ and $\psi_2$ both are trivial, then $M_1=1$
and $M_2$ is a prime number,
\item
 otherwise, $\psi_1$ and $\psi_2$ are primitive characters.
\end{enumerate}
Put $M=M_1 M_2$ and $\psi=\psi_1 \psi_2$. 

Let $f_k(Qz,\psi_1,\psi_2)$, where $QM_1M_2\vert N$ be the set of 
elements of $\mathcal{E}_k(N,\psi)$ as given in Theorems $4.7.1$ and 
$4.7.2$ of \cite{miyake} which form a basis of common eigenforms for 
all the Hecke operators $T_n$ of level $N$, with $(n,N)=1$.
\begin{rmk}\label{r1}
{\rm Using \thmref{zagier} and following the lines of Proposition $3$ of 
\cite{ghate1}, we have the following:\\
For positive integers $k$, $k_1$, $k_2$, $m$ satisfying $k_2\ge k_1+2>2$
and $k=k_1+k_2+2m$, $g \in S_{k_1}(N,\chi)$ an a.e. eigenform which is
a newform, $h= E_{k_2}^{(N, \psi)} \in \mathcal{E}_{k_2}(N, \psi)$,
if {\rm dim} $S_k^{new}(N,\chi\psi) \ge 2$ then $\left[g,h\right]_m$ 
is not an a.e. eigenform.}   
\end{rmk}
\begin{rmk}\label{r2}
{\rm Similarly as mentioned in the previous remark, we have an analogous
result to Proposition $4$ of \cite{ghate1} in this case:\\
For positive integers $k$, $k_1$, $k_2$, $m$ satisfying the same
condition as in the previous remark, $g=f_{k_1}(z,\chi_1\chi_2)$ 
an a.e. eigenform as described above with $\chi$ primitive, 
$h=E_{k_2}^{(N,\chi)}\in \mathcal{E}_{k_2}(N,\chi)$, if  
{\rm dim} $S_k^{new}(N,\chi\psi) \ge 2$ then $\left[g,h\right]_m$ 
is not an a.e. eigenform.}   
\end{rmk}
\begin{thm}\label{cusp-eis}
Let $k_1, k_2, k, m$ be positive integers such that $k=k_1+k_2+2m$ and 
let $N$ be square-free.
\begin{enumerate}
\item[(i)]
If $g \in S_{k_1}(\Gamma_1(N))$ and $h \in S_{k_2}(\Gamma_1(N))$ 
are a.e. eigenforms, then $\left[g,h\right]_m$ is not an a.e. eigenform.
\item[(ii)]
Let $k_1\ge 3$ and $k_2\ge k_1+2>2$.
Suppose that $g\in S_{k_1}(N,\chi)$ is an a.e. eigenform
which is a newform and $h\in \mathcal{E}_{k_2}(N,\psi)$. If 
{\rm dim} $S_k^{new}(N, \chi\psi)\ge 2$, then $\left[g,h\right]_m$ is not 
an a.e. eigenform.
\item[(iii)]
Let $k_1, k_2\ge 3$, $|k_1-k_2|\ge 2$. Let 
$g=f_{k_1}(z,\chi_1,\chi_2)\in \mathcal{E}_{k_1}(N,\chi)$ and 
$h=f_{k_2}(z,\psi_1,\psi_2)\in \mathcal{E}_{k_2}(N,\psi)$ be a.e. eigenforms 
as mentioned above with $\chi$ and $\psi$ primitive characters. If
{\rm dim} $S_k^{new}(N,\chi\psi) \ge 2$, $\left[g,h\right]_m$ is not an 
a.e. eigenform.  
\end{enumerate}
\end{thm}
\begin{proof}
Assume that 
\begin{equation}\label{W_Q}
\left[g,h\right]_m = f 
\end{equation}
is an a.e. eigenform. Then $f(z)=f_0(Qz)$,
where $M\vert N$, $Q\vert (N/M)$ and $f_0\in S_k(M,\chi)$ is a normalized
newform. Since $N$ is square-free, for any divisior $Q$ of $N$, $(Q,N/Q)=1$.
We also have  $f\vert W_Q = Q^{-k/2}\chi(w)f_0(z)$. Applying the operator
$W_Q$  to \eqref{W_Q} and by \lemref{commute}, we get 
$\left[g\vert W_Q, h\vert W_Q\right]_m=f\vert W_Q={\rm const}.~ f_0$. This 
gives a contradiction since the $q$-expansion of $f_0$ (being primitive) 
starts with $q$, whereas the $q$-expansion of 
$\left[g\vert W_Q, h\vert W_Q\right]_m $ starts with at least $q^2$. 
This proves (i).\\
Let $h= f_{k_2}(Qz,\psi_1,\psi_2)$, for $Q\vert(N/M_1M_2)$. 
Since $N$ is square-free, for any divisor $Q$ of $N$, we have 
$(Q,N/Q)=1$. Now applying
$W_{N/QM_2}$ on $h$ and using Proposition $1$ and Lemma $1$ of 
\cite{ghate1}, we get
$h\vert W_{N/QM_2}={\rm const}.~ f_{k_2}(\frac{Nz}{M_1M_2}, 
\psi_0, \overline\psi_1\psi_2)
={\rm const}.~ E_{k_2}^{(N,\overline\psi_1\psi_2)}$,
where $\psi_0$ is the principal character.
Now assume on the contrary that $\left[g,h\right]_m$ is an a.e. eigenform. 
Applying $W_{N/QM_2}$ to $\left[g,h\right]_m$ and using \lemref{commute},
we see that 
$\left[g\vert W_{N/QM_2}, h\vert W_{N/QM_2}\right]_m\in
S_k(N, \chi_{QM_2}\overline\chi_{N/QM_2}\overline\psi_1\psi_2)$ is an a.e.
eigenform. Since the $W$-operator is an isomorphism and takes a 
newform space to a newform space, 
{\rm dim} $S_k(N, \chi_{QM_2}\overline\chi_{N/QM_2}\overline\psi_1\psi_2)\ge 2$. 
Then applying \rmkref{r1} to $g\vert W_{N/QM_2} \in   
S_k(N, \chi_{QM_2}\overline\chi_{N/QM_2})$ and $h\vert W_{N/QM_2}
={\rm const}.~ E_{k_2}^{(N,\overline\psi_1\psi_2)}$, we get a contradiction.
This proves (ii).\\
If $k_2-k_1\ge 2$, then as in the proof of (ii), 
we apply the operator $W_{N/M_2}$ to $h$ and $g$ and we get 
$h\vert W_{N/M_2}= {\rm const}.~ E_{k_2}^{(N,\overline\psi_1\psi_2)}$ and 
$g\vert W_{N/M_2}$ is a form with primitive character. 
Applying \rmkref{r2}, we get (iii). If  
$k_1-k_2\ge 2$, then interchanging the roles of $g$ and $h$ 
gives the required result.
\end{proof} 
{\bf Acknowledgments}\\
I thank D. Lanphier for giving useful informations about his papers.
I thank B. Ramakrishnan for useful discussions and making numerous 
suggestions. I also thank Sanoli Gun for her useful comments.
Finally, I thank the referee for meticulously reading the manuscript
and making numerous suggestions which improved the presentation.

\end{document}